\documentclass[12pt]{amsart}
\usepackage[pdftex]{graphicx}
\usepackage{amsmath,amssymb,stmaryrd}
\usepackage{amsthm, amsfonts}
\usepackage{enumerate}
\usepackage{comment}
\usepackage[all]{xy}
\usepackage[pdftex]{color}
\usepackage{hyperref}
\usepackage[margin=27truemm]{geometry}
\usepackage{stmaryrd}
\SetSymbolFont{stmry}{bold}{U}{stmry}{m}{n}

\theoremstyle{plain}
\newtheorem{thm}{Theorem}[section]
\theoremstyle{definition}
\newtheorem{dfn}[thm]{Definition}
\newtheorem{rem}[thm]{Remark}
\theoremstyle{plain}
\newtheorem{lem}[thm]{Lemma}
\newtheorem{prop}[thm]{Proposition}
\newtheorem{cor}[thm]{Corollary}

\theoremstyle{definition}

\newcommand{\s}{\mathbf{S}}

\numberwithin{equation}{section}

\title{Quasi-isomorphism of grid chain complexes for a connected sum of knots}
\author{Hajime Kubota}
\date{}

\begin{document}

\begin{abstract}
We give a purely combinatorial proof of a K\"{u}nneth formula for the minus version of knot Floer homology of connected sums by constructing a quasi-isomorphism of grid chain complexes.
The quasi-isomorphism naturally deduces that the Legendrian and transverse invariants behave functorially with respect to the connected sum operation.
\end{abstract}

\maketitle

\section{Introduction}

Grid homology is a combinatorial reconstruction of knot Floer homology developed by Manolescu, Ozsv\'{a}th, Szab\'{o}, and Thurston \cite{oncombinatorial}.
It is an interesting problem whether a purely combinatorial proof of known results of knot Floer homology is given using grid homology.
For example, there are combinatorial reformulations of the knot Floer invariants such as the tau, epsilon, and Upsilon invariant \cite{grid-tau, grid-upsilon, A-combinatorial-description-of-the-knot-concordance-invariant-epsilon}.
Furthermore, the concordance invariance of the tau and Upsilon invariant are also proved purely combinatorially \cite{Grid-diagrams-and-the-Ozsvath-Szabo-tau-invariant, grid-upsilon-concordance}.

Throughout the paper, we will only consider knots in $S^3$ and work with the minus version of grid homology and knot Floer homology with $\mathbb{F}=\mathbb{Z}/2\mathbb{Z}$ coefficients.
The main purpose of this paper is to give a combinatorial proof of K\"{u}nneth formula for knot Floer homology of connected sums proved in \cite{Holomorphic-disks-and-knot-invariants},
\begin{equation}
\label{eq:kunneth}
HFK^-(K_1)\otimes_{\mathbb{F}[U]} HFK^-(K_2)\cong HFK^-(K_1\# K_2).
\end{equation}

The connected sum operation has not been treated in grid homology, despite being a basic operation for knots.
We remark that the standard textbook of grid homology by Ozsv\'{a}th, Stipsicz, and Szab\'{o} \cite{grid-book} does not contain this formula \eqref{eq:kunneth}.
The grid Legendrian and transverse invariants are combinatorial invariants defined using grid homology, but their additivity under connected sums was proved using knot Floer homology (See Remark \ref{rem:additivity-of-the-Legendrian}).

Recently, the author \cite{Grid-homology-for-spatial-graphs-and-a-Kunneth-formula-of-connected-sum} proved the K\"{u}nneth formula of the hat version using grid homology.
The proof was done using the extended grid homology for spatial graphs and thus was not completed within the framework of the knot grid homology.
In this paper, we will give quasi-isomorphism of the grid complex $GC^-$ and prove the K\"{u}nneth formula \eqref{eq:kunneth} using only knot grid homology, without the extended grid homology for spatial graphs.
Furthermore, we will quickly prove the additivity of the grid Legendrian and transverse invariants using our quasi-isomorphism.

\subsection{Main results}
First, we construct the grid diagram representing $K_1\# K_2$ from two grid diagrams representing $K_1$ and $K_2$.
Let $g_1$ and $g_2$ be two $(n-1)\times(n-1)$ grid diagrams representing two knots $K_1$ and $K_2$ respectively.
We can assume that $g_1$ has an $X$-marking in the top left square and $g_2$ bottom right square (see Figure \ref{fig:g1g2-to-g}).
Remove the top $X$-marking of $g_1$, attach one new row and column, and put new $X$-markings in the top right and bottom left square and new $O$-marking in the bottom right square. 
The resulting diagram is denoted by $g_1'$.
Take a new diagram $g_2'$ from $g_2$ similarly, but the new $O$-marking is in the top left square.
Construct a $2n\times2n$ grid diagram whose northwest $n\times n$ block is the same as $g_1'$ and southeast $n\times n$ block is the same as $g_2'$.
Switch the two $O$-markings in the $n$-th and $(n+1)$-th row.
Then we call the resulting grid diagram $g_\#$.
Clearly $g_\#$ represents $K_1\# K_2$.
We remark that the transformation from $g_1$ and $g_2$ to $g_1'$ and $g_2'$ respectively is done by a finite sequence of moves on grid diagrams corresponding to Reidemeister moves.

\begin{figure}
\centering
\includegraphics[scale=0.5]{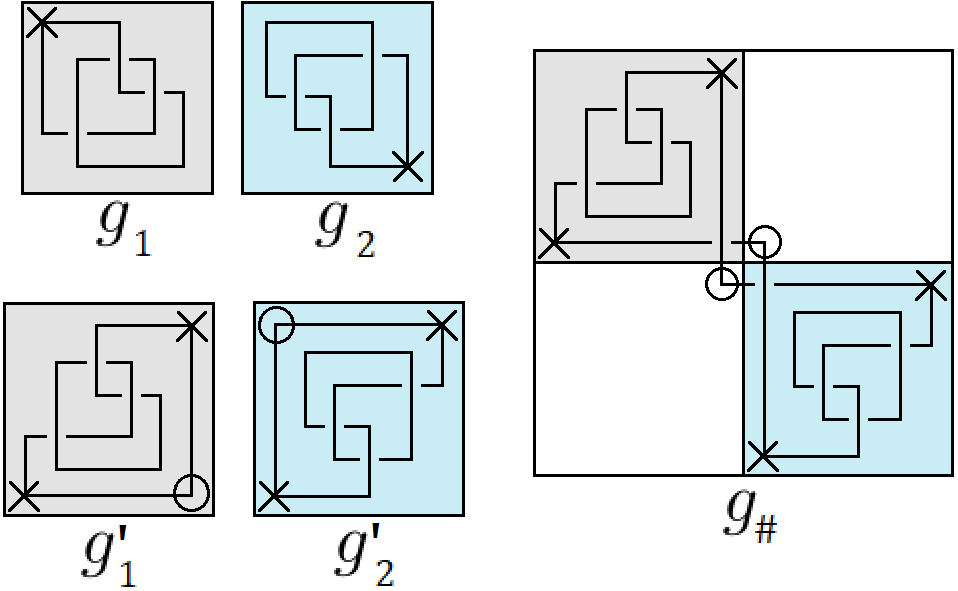}
\caption{An example of grid diagrams for the main theorem. Left: grid diagrams $g_1$, $g_2$ and stabilized grid diagrams $g_1'$, $g_2'$. Right: the combined grid diagram from $g_1'$, $g_2'$.}
\label{fig:g1g2-to-g}
\end{figure}

The grid chain complex $GC^-(g)$ for an $n\times n$ grid diagram $g$ is a bigraded chain complex over $\mathbb{F}[U_1,\dots,U_n]$, where $\mathbb{F}=\mathbb{Z}/2\mathbb{Z}$ and $U_i$'s are formal variables.
See Section \ref{sec:grid-homology-for-knots} or \cite{grid-book} for details.

\begin{thm}
\label{thm:main:connected}
Let $g_1, g_2$, and $g_\#$ be grid diagrams described above.
There are a subcomplex $C$ of $GC^-(g_\#)$ and two quasi-isomorphisms 
\[
C\to GC^-(g_\#)\hspace{5mm} \mathrm{and} \hspace{5mm} C\to GC^-(g_1)\otimes_{\mathbb{F}}GC^-(g_2)/U_1=U_{2n},
\]
where $GC^-(g_1), GC^-(g_2)$, and $GC^-(g_\#)$ are regarded as modules over $\mathbb{F}[U_1,\dots,U_{n-1}]$, $\mathbb{F}[U_{n+1},\dots,U_{2n}]$, and $\mathbb{F}[U_1,\dots,U_{2n}]$, and these maps are bigraded over $\mathbb{F}[U_1,\dots,U_{2n}]$ and over $\mathbb{F}[U_1,\dots,U_{n-1}, U_{n+1},\dots,U_{2n}]$ respectively.
\end{thm}

\begin{rem}
\begin{itemize}
    \item By letting $U_1=U_{2n}=U$, we have $\frac{GC^-(g_1)\otimes_{\mathbb{F}}GC^-(g_2)}{U_1=U_{2n}} \cong GC^-(g_1)\otimes_{\mathbb{F}[U]}GC^-(g_2)$ as $\mathbb{F}[U]$-modules.
Therefore this theorem implies the K\"{u}nneth formula of connected sum for the minus version of knot Floer homology.
    \item We will explicitly give the subcomplex $C$ and two quasi-isomorphisms in the above.
\end{itemize}

\end{rem}

In grid homology, there are two ways to define the tau invariant.
In this paper, we will use the following definition.
\begin{dfn}
For a knot $K$, $\tau(K)$ is $-1$ times the maximal integer $i$ for which there is a homogeneous, non-torsion element in $GH^-(K)$ with Alexander grading equal to $i$.
\end{dfn}
Sarker \cite[Theorem 3.3]{Grid-diagrams-and-the-Ozsvath-Szabo-tau-invariant} gave a combinatorial proof of the concordance invariance of the tau invariant.

\begin{cor}
\label{thm:tau-additivity}
\[
\tau(K_1\#K_2)=\tau(K_1)+\tau(K_2)
\]
for any two knots $K_1$ and $K_2$.
\end{cor}
\begin{proof}
It is a direct consequence of Theorem \ref{thm:main:connected}.
\end{proof}

Grid homology gives invariants for Legendrian and transverse knots in $(\mathrm{R}^3,\xi_{\mathrm{std}})$ \cite{grid-Legendrianknot}.
The invariant for Legendrian knots is called the grid Legendrian invariant or the GRID invariant.
There are examples of Legendrian knots with the topological type and the same classical invariants that are distinguished by the grid Legendrian invariant \cite[Section 8]{grid-Legendrianknot}.
In addition, this invariant provides obstructions to decomposable Lagrangian cobordisms, certain good cobordisms between two Legendrian knots \cite[Theorem 1.2]{Lagrangian-cobordisms-and-Legendrian-invariants-in-knot-Floer-homology}.
The grid transverse invariant is applied to study the transverse simplicity of topological knots \cite{Transverse-knots-distinguished-by-knot-Floer-homology}.
These invariants are extended to null-homologous Legendrian and transverse knots in general closed contact three-manifolds using knot Floer homology \cite{Heegaard-Floer-invariants-of-Legendrian-knots-in-contact-three-manifolds}.

The grid Legendrian invariants $\lambda^+,\lambda^-\in GH^-(\mathcal{K})$ are defined to be the homology classes determined by the canonical generators of the grid chain complex (Definition \ref{dfn:grid-Legendrian-inv}).
By the construction of the quasi-isomorphism of the above theorem, we can easily trace the generators and hence we can quickly show the additivity of the grid Legendrian invariant.
\begin{thm}
\label{thm:connected-Legendrian-invariant}
Let $g_1$, $g_2$, and $g_\#$ be grid diagrams representing a Legendrian knots $\mathcal{K}_1$, $\mathcal{K}_2$, and $\mathcal{K}_1\# \mathcal{K}_2$ respectively. 
Then there is an isomorphism
\[
GH^-(g_1)\otimes GH^-(g_2)\to GH^-(g_\#),
\]
which maps $\lambda^+(\mathcal{K}_1)\otimes\lambda^+(\mathcal{K}_2)$ to $\lambda^+(\mathcal{K}_1\#\mathcal{K}_2)$ and $\lambda^-(\mathcal{K}_1)\otimes\lambda^-(\mathcal{K}_2)$ to $\lambda^-(\mathcal{K}_1\#\mathcal{K}_2)$.
\end{thm}

The grid transverse invariant is defined to be the homology class $\lambda^+\in GH^-(g)$, so the following corollary is proved immediately.
\begin{cor}
\label{cor:connected-transverse-invariant}
Let $g_1$, $g_2$, and $g_\#$  be two good grid diagrams representing the transverse knots $\mathcal{T}_1$, $\mathcal{T}_2$, and $\mathcal{T}_1\# \mathcal{T}_2$ respectively. 
Then there is an isomorphism
\[
GH^-(g_1)\otimes GH^-(g_2)\to GH^-(g_\#)
\]
which maps $\theta^-(\mathcal{T}_1)\otimes \theta^-(\mathcal{T}_2)$ to $\theta^-(\mathcal{T}_1\#\mathcal{T}_2)$.
\end{cor}

\begin{rem}
\label{rem:additivity-of-the-Legendrian}
The behavior of the Legendrian and transverse invariants under connected sum operations is already known.
V\'{e}rtesi \cite[Corollaries 1.2 and 1.3]{Transversely-nonsimple-knots} showed the additivity of these invariants using knot Floer homology.
To prove it, she compared the cycles of knot Floer complexes corresponding to the canonical cycles of grid complexes.

On the other hand, the present paper gives a purely combinatorial proof of that fact.
This paper enables us to deal with connected sum operations in grid homology at the chain complex level.
\end{rem}

\subsection{Outline of the paper.}
In Section \ref{sec:grid-homology-for-knots}, we review grid homology for knots.
In Section \ref{sec:Legendrian-transverse-knots}, we briefly describe Legendrian and transverse knots and the relationship between them and grid homology.
In Section \ref{sec:preparation-proof}, we prepare some notations and review the destabilization map we will use to prove the main theorems.
In section \ref{sec:proof-of-main-thm}, we construct quasi-isomorphisms and prove Theorem \ref{thm:main:connected}.
In Section \ref{sec:Legendrian-invariant}, we quickly show Theorem \ref{thm:connected-Legendrian-invariant} using the quasi-isomorphisms in Section \ref{sec:proof-of-main-thm}.

\section{The definition of grid homology for knots}
\label{sec:grid-homology-for-knots}
This section quickly provides an overview of grid homology for knots and links.
Referring to \cite[Remark 4.6.13]{grid-book}, we will introduce the hat version without using the minus version.
See \cite{grid-book} for details.

A \textbf{planar grid diagram} $g$ (Figure \ref{fig:state}) is an $n\times n$ grid of squares some of which is decorated with an $X$- or $O$- marking such that it satisfies the following conditions.
\begin{enumerate}[(i)]
\item There is exactly one $O$ and $X$ on each row and column.
\item $O$'s and $X$'s do not share the same square.
\end{enumerate}
We denote the set of $O$-markings by $\mathbb{O}$ and the set of $X$-markings by $\mathbb{X}$.
We often use the labeling of markings as $\{O_i\}_{i=1}^n$ and $\{X_j\}_{j=1}^n$.

A grid diagram determines a link in $S^3$.
Drawing oriented segments from the $O$-markings to the $X$-markings in each row and oriented segments from the $X$-markings to the $O$-marking in each column.
Assume that the vertical segments always cross above the horizontal segments.
Then we can recover the link diagram.
In this paper, we only consider grid diagrams representing knots.

It is known that any two grid diagrams for the same knot can be connected by a finite sequence of moves on grid diagrams called grid moves; see \cite{Cromwell-gridmove,oncombinatorial}.

A \textbf{toroidal grid diagram} is a grid diagram that is regarded as a diagram on the torus obtained by identifying edges in a natural way.
We assume that every toroidal diagram is oriented in a natural way.
We write the horizontal circles and vertical circles which divide the torus into squares as $\boldsymbol{\alpha}=\{\alpha_i\}_{i=1}^n$ and $\boldsymbol{\beta}=\{\beta_j\}_{j=1}^n$ respectively.
A \textbf{state} $\mathbf{x}$ of a toroidal diagram $g$ is a bijection $\boldsymbol{\alpha}\rightarrow\boldsymbol{\beta}$.
We denote by $\s(g)$ the set of states of $g$.
We describe a state as $n$ points on a toroidal grid diagram (figure \ref{fig:state}).
\begin{figure}
\centering
\includegraphics[scale=0.4]{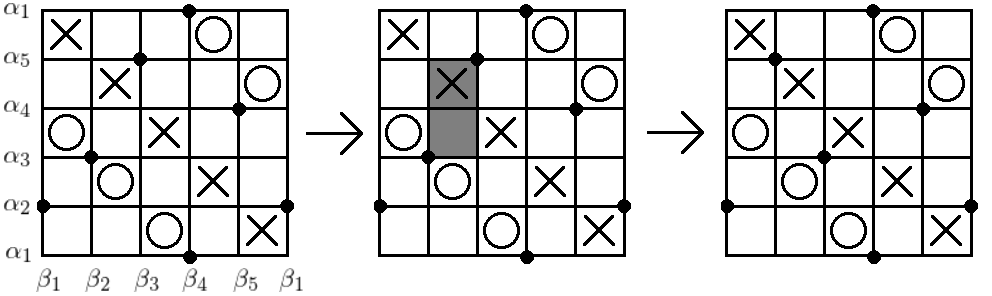}
\caption{an example of states and a rectangle}
\label{fig:state}
\end{figure}
There are $n\times n$ squares on $g$ that is separated by $\boldsymbol{\alpha}\cup\boldsymbol{\beta}$.
Fix $\mathbf{x},\mathbf{y}\in S$(g), a \textbf{domain} $p$ from $\mathbf{x}$ to $\mathbf{y}$ is a formal sum of the closure of squares satisfying $\partial(\partial_\alpha p)=\mathbf{y}-\mathbf{x}$ and $\partial(\partial_\beta p)=\mathbf{x}-\mathbf{y}$, where $\partial_\alpha p$ is the portion of the boundary of $p$ in the horizontal circles $\alpha_1\cup\dots\cup\alpha_n$ and $\partial_\beta p$ is the portion of the boundary of $p$ in the vertical ones.
A domain $p$ is \textbf{positive} if the coefficient of any square is non-negative.
Let $\pi(\mathbf{x},\mathbf{y})$ denote the set of domains from $\mathbf{x}$ to $\mathbf{y}$.
Throughout the paper, we always assume that a domain is positive.
For two domains $p_1\in\pi(\mathbf{x},\mathbf{y})$ and $p_2\in\pi(\mathbf{y},\mathbf{z})$, the \textbf{composite domain} $p_1*p_2$ is the domain from $\mathbf{x}$ to $\mathbf{z}$ such that the coefficient of each square is the sum of the coefficient of the square of $p_1$ and $p_2$.
Consider $\mathbf{x},\mathbf{y}\in \mathbf{S}(g)$ that coincide with $n-2$ points.
An \textbf{rectangle} $r$ from $\mathbf{x}$ to $\mathbf{y}$ is a domain such that $\partial r$ is the union of four segments.
A rectangle $r$ is \textbf{empty} if $\mathbf{x}\cap\mathrm{Int}(r)=\mathbf{y}\cap\mathrm{Int}(r)=\emptyset$.
Let $\mathrm{Rect}^\circ(\mathbf{x},\mathbf{y})$ be the set of empty rectangles from $\mathbf{x}$ to $\mathbf{y}$ if $\mathbf{x}$ and $\mathbf{y}$ coincide with $n-2$ points and $\mathrm{Rect}^\circ(\mathbf{x},\mathbf{y})=\emptyset$ otherwise.

\begin{dfn}
\label{dfn:minus-chain-cpx}
The minus version of \textbf{grid chain complex} $(GC^-(g),\partial^-)$ is the free $\mathbb{F}[U_1,\dots,U_n]$-module generated by $\s(g)$, where $\mathbb{F}=\mathbb{Z}/2\mathbb{Z}$.
The differential $\partial^-$ is the module map defined by for $\mathbf{x}\in\s(g)$,
\[
\partial^-(\mathbf{x})=\sum_{\mathbf{y}\in\s(g)}\left(
\sum_{\{r\in \mathrm{Rect}^\circ(\mathbf{x},\mathbf{y})|r\cap\mathbb{X}=\emptyset\}}
U_1^{O_1(r)}\cdots U_{n}^{O_{n}(r)}
\right)\mathbf{y},
\]
where $O_i(r)=1$ if $r$ contains $O_i$ and $O_i(r)=0$ otherwise for $i=1,\dots,n$.
\end{dfn}

There are two gradings for $GC^-(g)$, the \textbf{Maslov grading} and the \textbf{Alexander grading}.
A \textbf{planar realization} of a toroidal diagram $g$ is a planar figure obtained by cutting toroidal diagram $g$ along $\alpha_i$ and $\beta_j$ for some $i$ and $j$, and putting on $[0,n)\times[0,n)\in\mathbb{R}^2$ in a natural way.
For two points $(a_1,a_2),(b_1,b_2)\subset\mathbb{R}^2$, let $(a_1,a_2)<(b_1,b_2)$ if $a_1<b_1$ and $a_2<b_2$.
For two sets of finitely many points $A,B\subset\mathbb{R}^2$, let $\mathcal{I}(A,B)$ be the number of pairs $a\in A,b\in B$ with $a<b$ and let $\mathcal{J}(A,B)=(\mathcal{I}(A,B)+\mathcal{I}(B,A))/2$.

Then for $\mathbf{x}\in \mathbf{S}(g)$, the Maslov grading $M(\mathbf{x})$ and the Alexander grading $A(\mathbf{x})$ are defined by
\begin{align}
M(\mathbf{x})&=\mathcal{J}(\mathbf{x}-\mathbb{O},\mathbf{x}-\mathbb{O})+1,
\label{mm}
\\
A(\mathbf{x})&=\mathcal{J}(\mathbf{x},\mathbb{X}-\mathbb{O})+\frac{1}{2}\mathcal{J}(\mathbb{O}+\mathbb{X},\mathbb{O}-\mathbb{X})-\frac{n-1}{2}.
\label{aa}
\end{align}
These two gradings are extended to the whole $GC^-(g)$ by
\begin{align}
\label{uu}
M(U_i)=-2,\ A(U_i)=-1\ (i=1,\dots,n-1).
\end{align}
It is known that these two gradings are independent of the choice of the planar realization and that the differential $\partial^--$ drops the Maslov grading by 1 and preserves or drops the Alexander grading.
So $(CF^-(g),\partial^-)$ is an absolute Maslov graded, Alexander graded chain complex \cite[Sections 4.3 and 4.6]{grid-book}.

For any $1\leqq i,j\leqq n$, multiplication by $U_i$ is chain homotopic to multiplication by $U_j$ \cite[Lemma 4.6.9]{grid-book}.
So the minus version of grid homology is defined as follows.
\begin{dfn}
The minus version of \textbf{grid homology} of $g$, denoted $GH^-(g)$, is the homology of $(CF^-(g),\partial^-)$, thought of as a bigraded $\mathbb{F}[U]$-module, where the action of $U$ is induced by multiplication by $U_i$.
\end{dfn}
\begin{thm}[{\cite[Theorem 5.3.1]{grid-book}}]
Let $g_1$ and $g_2$ be two grid diagrams representing the same knot $K$.
Then $GH^-(g_1)\cong GH^-(g_2)$ as bigraded $\mathbb{F}[U]$-modules.
\end{thm}
We will denote by $GH^-(K)$ if there is no confusion.

\section{Legendrian and transverse knots in grid homology}
\label{sec:Legendrian-transverse-knots}
\subsection{Legendrian and transverse knots}
We briefly review the Legendrian and transverse knot and the Legendrian and transverse invariants with grid homology.
See \cite{Legendrian-and-transversal-knots,grid-Legendrianknot} for details.

In this paper, a \textbf{Legendrian knot} is a smooth knot $\mathcal{K}\subset\mathbb{R}^3$ whose tangent vectors are contained in the contact planes of $\mathrm{Ker}\,\xi_{\mathrm{std}}$, where $\xi_{\mathrm{std}}=dz-y\, dx$.
Two knots are Legendrian isotopic if they are connected by a smooth one-parameter family of Legendrian knots.
There are two classical invariants of a Legendrian knot $\mathcal{K}$; the \textbf{Thurston-Bennequin number} $\mathrm{tb}(\mathcal{K})$ and the \textbf{rotation number} $\mathrm{r}(\mathcal{K})$.
See for \cite{Legendrian-and-transversal-knots} for the definition.
We can realize these invariants using grid diagrams \cite[Theorem 1.1]{grid-Legendrianknot}.
Legendrian knots are studied via their front projections to the $xz$-plane by the projection map $(x,y,z)\mapsto(x,z)$.
The front projection of a Legendrian knot has no vertical tangencies and in the generic case, its singularities are double points and cusps.
A \textbf{transverse knot} $\mathcal{T}$ is a knot whose tangent vectors are transverse to the contact planes of $\mathrm{Ker}\,\xi_{\mathrm{std}}$.
Two transverse knots are transverse isotopic if they are connected by a smooth one-parameter family of transverse knots.
Given a Legendrian knot $\mathcal{K}$, there are transverse knots $\mathcal{T}$ that are arbitrarily close to $\mathcal{K}$ (in $C^1$ topology).
Furthermore, there is a neighborhood of $\mathcal{K}$ such that any two transverse knots in this neighborhood are transverse isotopic.
The \textbf{transverse push-off} $\mathcal{T}(\mathcal{K})$ of $\mathcal{K}$ is the transverse knot type which can be represented by a transverse knot arbitrarily close to $\mathcal{K}$.
Any transverse knot $\mathcal{T}$ is transversely isotopic to the transverse push-off of some Legendrian knot $\mathcal{K}$.
Then $\mathcal{K}$ is called \textbf{Legendrian approximation} of $\mathcal{T}$.

\begin{figure}
\centering
\includegraphics[scale=0.45]{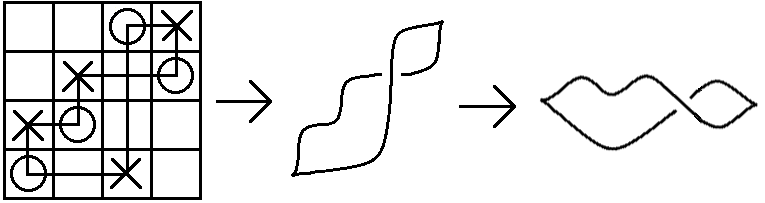}
\caption{The Legendrian knot determined by a grid diagram}
\label{fig:frontprojection}
\end{figure}

\subsection{the Legendrian and transverse grid invariant}
There is an algorithm for giving the front projection of a Legendrian knot from a (planar) grid diagram (Figure \ref{fig:frontprojection}. See \cite[Section 4]{grid-Legendrianknot} and \cite[Section 12.2]{grid-book} for details).
Let $g$ be a planar grid diagram for a knot $K$.
Connect each marking of $g$ by horizontal and vertical segments to obtain a diagram of $K$ on $g$.
Rotate $g$ by $45^\circ$ clockwise.
Turn each corner of the diagram of $K$ on $g$ into cusps or smooth arcs.
Switch all crossings of the diagram.
Then we get the front projection of the Legendrian knot whose topological type is $m(K)$.
We say that a grid diagram $g$ represents the Legendrian knot $\mathcal{K}$ if $\mathcal{K}$ is obtained from $g$ by the above procedure.
Furthermore, considering the transverse push-off of $\mathcal{K}$, a grid diagram uniquely specified a transverse knot.
Any transverse knot is represented by grid diagrams by taking its Legendrian approximation.

In \cite{grid-Legendrianknot}, Ozsv\'{a}th, Szab\'{o}, and Thurston defined \textbf{Legendrian and transverse grid invariants} using grid homology.
For a grid diagram $g$, let $\mathbf{x}^+(g)$ be the canonical state consisting of the northeast corners of squares decorated by $X$.
Similarly, let $\mathbf{x}^-(g)$ be the canonical state consisting of the southwest corners of the same squares.

\begin{thm}[{\cite[Theorem 1.1]{grid-Legendrianknot}}]
\label{thm:invariance-lambda}
Let $\mathcal{K}$ be a Legendrian knot represented by a grid diagram $g$.
The homology classes in $GH^-(g)$ represented by $\mathbf{x}^+(g)$ and $\mathbf{x}^-(g)$, denoted by $\lambda^+(g)$ and $\lambda^-(g)$ respectively, are invariants of Legendrian knots; i.e., if $g$ and $g'$ represent Legendrian isotopic knots, then there is a quasi-isomorphism
\[
\Phi\colon GC^-(g)\to GC^-(g')
\]
such that $\Phi(\lambda^+(g))=\lambda^+(g')$ and $\Phi(\lambda^-(g))=\lambda^-(g')$.
\end{thm}

\begin{thm}[{\cite[Corollary 1.4]{grid-Legendrianknot}}]
\label{thm:invariance-theta}
Let $\mathcal{K}$ be a Legendrian knot represented by a grid diagram $g$, and $\mathcal{T}$ be its transverse push-off.
The homology class $\lambda^+(g)\in GH^-(g)$, denoted by $\theta(\mathcal{T})$, is an invariant of transverse knot; i.e., if $g$ and $g'$ represent two Legendrian approximations to $\mathcal{T}$, then there is a quasi-isomorphism
\[
\Psi\colon GC^-(g)\to GC^-(g')
\]
such that $\Psi(\theta(\mathcal{T}))=\theta(\mathcal{T'})$.
\end{thm}

\section{Preparations of the proof of the main theorem}
\label{sec:preparation-proof}

\subsection{The classification of the states}
\label{subsec:states-of-g}
The basic idea is the same as \cite[Section 5.1]{Grid-homology-for-spatial-graphs-and-a-Kunneth-formula-of-connected-sum}.
Let $\mathbf{g}_{11}=(\alpha_{n+1}\cup\dots\cup\alpha_{2n})\cap(\beta_1\cup\dots\cup\beta_n)$ be the $n^2$ points in $g_{11}$.
Similarly, define $\mathbf{g}_{12}=(\alpha_{n+1}\cup\dots\cup\alpha_{2n})\cap(\beta_{n+1}\cup\dots\cup\beta_{2n})$, $\mathbf{g}_{21}=(\alpha_1\cup\dots\cup\alpha_n)\cap(\beta_1\cup\dots\cup\beta_n)$, and $\mathbf{g}_{22}=(\alpha_1\cup\dots\cup\alpha_n)\cap(\beta_{n+1}\cup\dots\cup\beta_{2n})$ similarly. 
For a state $\mathbf{x}\in\s(g_{\#})$ and $i,j\in\{1,2\}$, let $\mathbf{x}_{ij}=\mathbf{x}\cap\mathbf{g}_{ij}$.
We will write each state $\mathbf{x}\in \mathbf{S}(g_{\#})$ uniquely as
\[
\mathbf{x}=
\begin{pmatrix}
\mathbf{x}_{11} & \mathbf{x}_{12} \\
\mathbf{x}_{21} & \mathbf{x}_{22} \\
\end{pmatrix}
,
\]
such that $\#\mathbf{x}_{11}=\#\mathbf{x}_{22}$ and $\#\mathbf{x}_{12}=\#\mathbf{x}_{21}$.
Then for $k=0,\dots,n$, let $\s_k(g_{\#})$ be the subset of $\s(g_{\#})$ defined by
\[
\s_k(g_{\#}) =\{\mathbf{x}\in\mathbf{S}(g_{\#})|\#\mathbf{x}_{12}=\#\mathbf{x}_{21}=k\}
\]
for $k=0,\dots,n$.

Take four points $a=\alpha_{1}\cap\beta_{1}, b=\alpha_{1}\cap\beta_{n+1}, c=\alpha_{n+1}\cap\beta_{1}$, and $d=\alpha_{n+1}\cap\beta_{n+1}$ on $g_\#$.
Let $\mathbf{AD}_1(g_\#)$ and $\mathbf{BC}_0(g_\#)$ be the subset of $\s_1(g_\#)$ and $\s_0(g_\#)$ defined by
\begin{align*}
\mathbf{AD}_1(g_\#) &=\{\mathbf{x}\in\s_1(g_\#)|a,d\in\mathbf{x}\},\\
\mathbf{BC}_0(g_\#) &=\{\mathbf{x}\in\s_0(g_\#)|b,c\in\mathbf{x}\}.
\end{align*}

\subsection{Stabilization and destabilization maps}
\label{subsed:stabi-destabi}
This section describes chain homotopy equivalence corresponding to destabilization maps constructed by \cite[Section 5.7]{grid-book}.
We will discuss the cases of destabilization of type $X\colon SE$ and $X\colon NW$.

Let $g'$ be an $n\times n$ grid diagram obtained from $g$ by stabilization of type $X\colon SE$.
Number the $O$-markings in $g'$ such that $O_1$ is in the new column and $O_2$ is in the next column (Figure 3).
Let $X_1$ be the $X$-marking in the same row as $O_1$ and $X_2$ be the $X$-marking in the same column as $O_1$.
Then we think that $GC^-(g)$ is the chain complex over $\mathbb{F}[U_2,\dots,U_n]$ and $GC^-(g')$ is over $\mathbb{F}[U_1,\dots,U_n]$.
\begin{figure}
    \centering
    \includegraphics[scale=0.7]{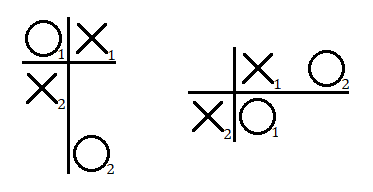}
    \caption{Left: the labeling for the case $X\colon SE$. Right: the labeling for the case $X\colon NW$.}
    \label{fig:Label-stabi}
\end{figure}

Take a point $c$ on $g'$ as the southeast corner of the square containing $O_1$.
Let $\mathbf{I}(g')$ be the set of states containing $c$ and $\mathbf{N}(g')$ be $\s(g')\setminus\mathbf{I}(g')$.
Consider the corresponding $\mathbb{F}[U_1,\dots,U_n]$-module splitting $GC^-(g')\cong \mathbf{I}\oplus\mathbf{N}$, where $\mathbf{I}$ and $\mathbf{N}$ are the submodule generated by states of $\mathbf{I}(g')$ and $\mathbf{N}(g')$ respectively.
Because of two $X$-markings adjacent to $O_1$, $\mathbf{I}$ is the subcomplex of $GC^-(g')$.
Thus we can write the differential of $GC^-(g')$ as a matrix
\[
\partial^-=\begin{pmatrix}
    \partial_\mathbf{I}^\mathbf{I} & \partial_\mathbf{N}^\mathbf{I}  \\
    0 & \partial_\mathbf{N}^\mathbf{N},
\end{pmatrix}
\]
where, for example, $\partial_\mathbf{I}^\mathbf{I}$ is the chain map counting the rectangles from a state of $\mathbf{I}(g')$ to a state of $\mathbf{I}(g')$.
Then $GC^-(g')$ is the mapping cone $\mathrm{Cone}(\partial_\mathbf{N}^\mathbf{I}\colon(\mathbf{N},\partial_\mathbf{N}^\mathbf{N})\to(\mathbf{I},\partial_\mathbf{I}^\mathbf{I}))$.

\begin{lem}[{\cite[Lemma 5.2.18]{grid-book}}]
\label{lem:e:I-C}
The natural correspondence between the state of $GC^-(g)$ and $\mathbf{I}(g')$ induces an isomorphism $e\colon (\mathbf{I},\partial_\mathbf{I}^\mathbf{I})\to GC^-(g)[U_1]$ of bigraded chain complexes over $\mathbb{F}[U_1,\dots,U_n]$.
\end{lem}

\begin{dfn}[{\cite[Definition 5.7.1]{grid-book}}]
An \textbf{empty hexagon} is a domain $h\in\pi(\mathbf{x},\mathbf{y})$ with $\mathbf{x}\in\mathbf{I}(g')$ and $\mathbf{y}\in\mathbf{N}(g')$ satisfying followings:
\begin{itemize}
    \item the local multiplicities of $h$ are all $0$ or $1$,
    \item $h$ has six corners, five $90^\circ$ corners and one $270^\circ$ corner which is at $c$,
    \item $\mathbf{x}\cap \mathrm{Int}(h)=\emptyset$.
\end{itemize}
We denote by $\mathrm{Hex}(\mathbf{x},\mathbf{y})$ the set of empty hexagons from $\mathbf{x}$ to $\mathbf{y}$.
\begin{figure}
    \centering
    \includegraphics[scale=0.7]{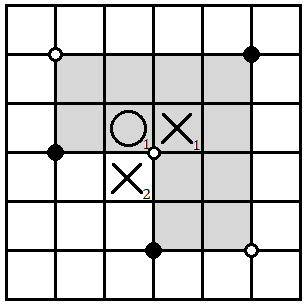}
    \caption{An empty hexagon counted by $\mathcal{H}_{Hex}$.}
    \label{fig:hexagon}
\end{figure}
\end{dfn}
Figure \ref{fig:hexagon} shows an example of an empty hexagon.

Define $\mathbb{F}[U_1,\dots,U_n]$-module maps $\mathcal{H}_{O_1}, \mathcal{H}_{Hex}\colon \mathbf{N}\to\mathbf{I}$with
\begin{align*}
\mathcal{H}_{O_1}(\mathbf{x}) &=\sum_{\mathbf{y}\in\mathbf{I}(g')}\left( \sum_{\{r\in\mathrm{Rect}^\circ(\mathbf{x},\mathbf{y})|O_1\in r,r\cap\mathbb{X}=\emptyset\}}U_2^{O_2(r)}\cdots U_n^{O_n(r)} \right)\mathbf{y}, \\
\mathcal{H}_{Hex}(\mathbf{x}) &=\sum_{\mathbf{y}\in\mathbf{I}(g')}\left( \sum_{\{r\in\mathrm{hex}(\mathbf{x},\mathbf{y})|O_1\in h,h\cap\mathbb{X}=\{X_1\}\}}U_2^{O_2(r)}\cdots U_n^{O_n(r)} \right)\mathbf{y}.
\end{align*}

\begin{dfn}
The \textbf{destabilization map} $\mathcal{D}_{SE}\colon GC^-(g') \to \mathrm{Cone}(U_1-U_2\colon GC^-(g)[U_1]\to GC^-(g)[U_1])$ of type $X\colon SE$ is defined by the following diagram:
\begin{equation}
\label{diagram:desta}
\xymatrix@=40pt{
\mathbf{N} \ar[r]^-{\partial_\mathbf{N}^\mathbf{I}} \ar[d]^-{e\circ \mathcal{H}_{O_1}} \ar[dr]^-{e\circ \mathcal{H}_{Hex}} & \mathbf{I} \ar[d]^-{e} \\
GC^-(g)[U_1] \ar[r]^-{U_1-U_2} & GC^-(g)[U_1]
}
\end{equation}
\end{dfn}
\begin{rem}
\begin{itemize}
    \item In \cite[Section5.7]{grid-book}, the stabilization map $\mathcal{S}_{SE}$ is defined, and  $\mathcal{D}_{SE}$ and $\mathcal{S}_{SE}$ are chain homotopy inverses to each other.
    See \cite[Exercise 5.7.2]{grid-book} for detailed discussion.
    \item In the above definitions, $U_1-U_2$ comes from the fact that $O_1$ and $O_2$ are in the same row and column respectively as $X_2$.
\end{itemize}
\end{rem}

For the case of destabilization of type $X\colon NW$, label the markings so that $O_1$ is in the stabilized $2\times 2$ block, $X_1$ is in the same column as $O_1$, $X_2$ in the same row as $O_1$, and $O_2$ is in the same row as $X_1$.
Then define $\mathcal{H}_{O_1}$ and $\mathcal{H}_{Hex}$ by the same equations above.
We will define stabilization and destabilization maps using these maps.
Let $\mathcal{D}_{NW}\colon GC^-(g') \to \mathrm{Cone}(U_1-U_2\colon GC^-(g)[U_1]\to GC^-(g)[U_1])$ be the destabilized map.

\section{The proof of the main theorems}
\label{sec:proof-of-main-thm}
\subsection{The subcomplex \texorpdfstring{$C$}{C} of \texorpdfstring{$GC^-(g_\#)$}{GC-(g)}}
We will construct the subcomplex $C$ of $GC^-(g_\#)$ stated in Theorem \ref{thm:main:connected} and prove that the inclusion $C\to GC^-(g_\#)$ is a quasi-isomorphism.
The key idea is that the quotient complex $GC^-(g_\#)/C$ consists of many acyclic parts in some sense.
This idea is almost the same as \cite[Lemmas 5.6 and 5.7]{Grid-homology-for-spatial-graphs-and-a-Kunneth-formula-of-connected-sum}.

\begin{dfn}
\label{dfn:C}
Let $C$ be the free $\mathbb{F}[U_1,\dots,U_{2n}]$-module generated by $\mathbf{AD}_1(g_\#)\cup\s_0(g_\#)$.
\end{dfn}

\begin{prop}
\label{prop:GC/C-is-acyclic}
$GC^-(g_\#)/C$ is acyclic.
\end{prop}
\begin{proof}
For $i=0,\dots,n$, let $S_i$ be a free $\mathbb{F}[U_1,\dots,U_{2n}]$-module generated by $\mathbf{S}_i(g_\#)$.
Then there is a sequence of subcomplexes
\[
S_0\subset(S_0\oplus S_1)\subset\dots\subset(S_0\oplus\dots\oplus S_n)=GC^-(g_\#).
\]

Consider the quotient complex $(S_0\oplus\dots\oplus S_n)/(S_0\oplus S_1)$.
This complex is finite in each Alexander grading.
Therefore the same argument as \cite[Lemma 5.6]{Grid-homology-for-spatial-graphs-and-a-Kunneth-formula-of-connected-sum} shows that it is acyclic.

We will verify that $(S_0\oplus S_1)/C$ is acyclic, which can be shown in the same way as \cite[Lemma 5.7]{Grid-homology-for-spatial-graphs-and-a-Kunneth-formula-of-connected-sum}.
Consider a state $\mathbf{x}=\begin{pmatrix}
\mathbf{x}_{11} & \mathbf{x}_{12} \\
\mathbf{x}_{21} & \mathbf{x}_{22} \\
\end{pmatrix}\in\s(g_\#)\setminus\mathbf{AD}_1(g_\#)$.
First we will group the states of $\s(g_\#)\setminus\mathbf{AD}_1(g_\#)$.
\begin{enumerate}[(1)]
    \item If $a\in\mathbf{x}$ and $d\notin\mathbf{x}$, then $\mathbf{x}_{11}$ contains a point on $\alpha_{n+1}$ or $\mathbf{x}_{22}$ contains a point on $\beta_{N+1}$.
    We can take a grouping consisting of two or four states in the sense of \cite[Lemma 5.7]{Grid-homology-for-spatial-graphs-and-a-Kunneth-formula-of-connected-sum}.
    Make one grouping for each state containing $a$.
    We remark that one of the states forming a grouping contains $a$.
    \item If $d\in\mathbf{x}$ and $a\notin\mathbf{x}$, then $\mathbf{x}_{11}$ contains a point on $\alpha_{1}$ or $\mathbf{x}_{22}$ contains a point on $\beta_{1}$.
    We can take a grouping consisting of two or four states.
    Make one grouping for each state containing $d$.
    We remark that one of the states forming a grouping contains $d$.
    \item The remaining states also form groups consisting of four states.
\end{enumerate}
Since $(S_0\oplus S_1)/C$ is finite in each Alexander grading, we can choose a grouping consecutively so that the chosen grouping is the subcomplex of the quotient complex of $(S_0\oplus S_1)/C$ by already selected groupings.
\cite[Lemma 5.7]{Grid-homology-for-spatial-graphs-and-a-Kunneth-formula-of-connected-sum} shows that each chosen grouping is acyclic.
Since $(S_0\oplus S_1)/C$ is finite in each Alexander grading, repeating this procedure completes the proof. 
\end{proof}

\subsection{The quasi-isomorphism \texorpdfstring{$C\to GC^-(g_1)\otimes_{\mathbb{F}}GC^-(g_2)/U_1=U_{2n}$}{}}
\label{subsec:quasi-f}
Let $C_0$ be the subcomplex of $C$ freely generated by $\s_0(g_\#)$.
Let $AD$ is a free $\mathbb{F}[U_1,\dots,U_{2n}]$-module generated by $\mathbf{AD}_1(g_\#)$.
Then $C$ is written as a mapping cone $\mathrm{Cone}(f\colon AD\to C_0)$, where $f$ counts the rectangles from a state of $AD$ to a state of $C_0$.
$f$ is obviously injective and thus there is a quasi-isomorphism $\mathrm{Cone}(f)\to C_0/\mathrm{Im}f$ (Lemma \ref{lem:conef-c/f}).
The rectangle $r\in\mathrm{Rect}^\circ(\mathbf{x},\mathbf{y})$ counted by the differential of $C_0$ does not contain $O_n$ but may contain $O_{n+1}$.
In addition, its corner points $(\mathbf{x}\cup\mathbf{y})\setminus(\mathbf{x}\cap\mathbf{y})$ are contained in either $g_{11}$ or $g_{22}$.
Hence we have an isomorphism of chain complexes over $\mathbb{F}[U_1,\dots,U_{2n}]$-module
\[
C_0 \cong GC^-(g_1')\otimes_{\mathbb{F}[U_{n+1}]}GC^-(g_2')[U_n],
\]
where we regard $GC^-(g_1')$ and $GC^-(g_2')$ as chain complexes over $\mathbb{F}[U_1,\dots,U_{n-1},U_{n+1}]$ and $\mathbb{F}[U_{n+1},\dots,U_{2n}]$ respectively, and $GC^-(g_1')\otimes_{\mathbb{F}[U_{n+1}]}GC^-(g_2')$ is considered to be a chain complex over $\mathbb{F}[U_1,\dots,U_{n-1},U_{n+1},\dots,U_{2n}]$.
On the other hand, in fact, $\mathrm{Im}f$ is isomorphic to $(U_n+U_{n+1})\cdot GC^-(g_1)\otimes_{\mathbb{F}}GC^-(g_2)[U_n,U_{n+1}]$.
To compare $C_0$ and $\mathrm{Im}f$, we will give a representation of $C_0$ using $g_1$ and $g_2$ rather than $g_1'$ and $g_2'$ using the destabilization map $\mathcal{D}$.

\begin{lem}[{\cite[Lemma A.3.9]{grid-book}}]
\label{lem:conef-c/f}
If $f\colon D\to D'$ is an injective chain map for bigraded chain complexes, then the map $\mathrm{Cone}(f)\to D'/f(D)$, $(d,d')\mapsto q(d')$ is a quasi-isomorphism, where $q\colon D'\to D'/f(D)$ is the projection.
\end{lem}

Let $\mathbf{II}(g_\#), \mathbf{IN}(g_\#), \mathbf{NI}(g_\#)$, and $\mathbf{NN}(g_\#)$ be the subset of $\s_0(g_\#)$ defined by
\begin{align*}
\mathbf{II}(g_\#) &=\{\mathbf{x}\in\s_0(g_\#)|b,c\in\mathbf{x}\},\\
\mathbf{IN}(g_\#) &=\{\mathbf{x}\in\s_0(g_\#)|b\in\mathbf{x},c\notin\mathbf{x}\},\\
\mathbf{NI}(g_\#) &=\{\mathbf{x}\in\s_0(g_\#)|b\notin\mathbf{x},c\in\mathbf{x}\},\\
\mathbf{NN}(g_\#) &=\{\mathbf{x}\in\s_0(g_\#)|b,c\notin\mathbf{x}\},\\
\end{align*}
and let $II,IN,NI$ and $NN$ be the free $\mathbb{F}[U_1,\dots,U_{2n}]$-module generated by each of them respectively.
Let $\mathbf{I}_{g_1'}$ and $\mathbf{N}_{g_1'}$ be the chain complexes defined in Section $\ref{subsed:stabi-destabi}$ with $g'=g_1'$.
By replacing $U_{n}$ with $U_{n+1}$, we regard them as a chain complex over $\mathbb{F}[U_1,\dots,U_{n-1},U_{n+1}]$-module.
Let $\mathbf{I}_{g_2'}$ and$\mathbf{N}_{g_2'}$ be the the chain complexes with $g'=g_2'$ similarly.
These are chain complexes $\mathbb{F}[U_{n+1},\dots,U_{2n}]$-module.
Then we can naturally identify of chain complexes over $\mathbb{F}[U_1,\dots,U_{2n}]$-module
\begin{align}
\label{eq:II-IxI}
II &\cong (\mathbf{I}_{g_1'}\otimes_{\mathbb{F}[U_{n+1}]}\mathbf{I}_{g_2'})[U_n],\\
IN &\cong (\mathbf{I}_{g_1'}\otimes_{\mathbb{F}[U_{n+1}]}\mathbf{N}_{g_2'})[U_n],\\
NI &\cong (\mathbf{N}_{g_1'}\otimes_{\mathbb{F}[U_{n+1}]}\mathbf{I}_{g_2'})[U_n],\\
NN &\cong (\mathbf{N}_{g_1'}\otimes_{\mathbb{F}[U_{n+1}]}\mathbf{N}_{g_2'})[U_n].
\end{align}

Recall that $C_0\cong GC^-(g_1')\otimes_{\mathbb{F}[U_{n+1}]}GC^-(g_2')[U_n]$.
We will apply the destabilization map for $GC^-(g_1')$ and $GC^-(g_2')$.
For abbreviation, we write $G$ instead of $GC^-(g_1)\otimes_{\mathbb{F}}GC^-(g_2)[U_n,U_{n+1}]$.
Let $\mathrm{Cone}(G)$ be the mapping cone of the following commutative diagram
\begin{equation}
\label{diagram:G}
\xymatrix@=50pt{
G \ar[r]^{U_{n+1}-U_{2n}} \ar[d]_-{U_1-U_{n+1}} & G \ar[d]^-{U_1-U_{n+1}}\\
G \ar[r]^-{U_{n+1}-U_{2n}}  & G.
}
\end{equation}
Consider the combined destabilization map
\begin{align*}
\mathcal{D}_{SE}\otimes \mathcal{D}_{NW} \colon C_0\cong GC^-(g_1')\otimes_{\mathbb{F}[U_{n+1}]}GC^-(g_2')[U_n] \to \mathrm{Cone}(G).
\end{align*}
This map is understood as the diagram in Figure \ref{diagram:desta-desta}.
\begin{figure}
\[
\xymatrix@=110pt{
NN\ar[rrr]^-{\mathrm{id}\otimes\partial_\mathbf{N}^\mathbf{I}} \ar[ddd]_-{\partial_\mathbf{N}^\mathbf{I}\otimes\mathrm{id}}|(.18)\hole|(.19)\hole|(.33)\hole|(.35)\hole \ar[dr]_-{(e\circ \mathcal{H}_{O_1})\otimes(e\circ \mathcal{H}_{O_1})}\ar[rrd]^-{(e\circ \mathcal{H}_{O_1})\otimes(e\circ \mathcal{H}_{Hex})}\ar[ddr]_-{(e\circ \mathcal{H}_{Hex})\otimes(e\circ \mathcal{H}_{O_1})}|(.26)\hole|(.29)\hole & {} & {} & NI\ar[ddd]^-{\partial_\mathbf{N}^\mathbf{I}\otimes\mathrm{id}}\ar[dl]_-{(e\circ \mathcal{H}_{O_1})\otimes e}\ar[ddl]^-{(e\circ \mathcal{H}_{Hex})\otimes e}\\
{} & G\ar[r]_-{U_{n+1}-U_{2n}}\ar[d]^-{U_1-U_{n+1}} & G\ar[d]_-{U_1-U_{n+1}} & {} \\
{} & G\ar[r]^-{U_{n+1}-U_{2n}} & G & {} \\
IN\ar[rrr]_-{\mathrm{id}\otimes\partial_\mathbf{N}^\mathbf{I}} \ar[ur]^-{e\otimes(e\circ \mathcal{H}_{O_1})}\ar[rru]_-{(e\circ \mathcal{H}_{Hex})\otimes e} & {} & {} & II\ar[ul]^-{e\circ e}
}
\]
\caption{A commutative diagram of the combined destabilization map}
\label{diagram:desta-desta}
\end{figure}
This diagram contains four diagrams the same as \eqref{diagram:desta} between the outside and inside squares.
$\mathcal{D}_{SE}$ and $\mathcal{D}_{NW}$ are chain homotopy equivalences and so is $\mathcal{D}_{SE}\otimes \mathcal{D}_{NW}$.

\begin{lem}
\label{lem:Imf-CxC}
We have an isomorphism $e\circ e\colon \mathrm{Im}(f)=(U_n+U_{n+1})\cdot II \cong (U_n+U_{n+1})\cdot GC^-(g_1)\otimes_{\mathbb{F}}GC^-(g_2)[U_n,U_{n+1}]$.
\end{lem}
\begin{proof}
Recall that the chain map $f$ counts rectangles from a state of $AD$ to a state of $C_0$.
$f$ always counts exactly two rectangles whose two of four corners are $a$ and $d$.
So we have $\mathrm{Im}(f)=(U_n+U_{n+1})\cdot II$.
By \eqref{eq:II-IxI} and Lemma \ref{lem:e:I-C}, the map $e\circ e\colon II \cong GC^-(g_1)\otimes_{\mathbb{F}}GC^-(g_2)[U_n,U_{n+1}]$ is an isomorphism of bigraded chain complexes.
\end{proof}

\begin{proof}[proof of Theorem \ref{thm:main:connected}]
By Lemma \ref{lem:conef-c/f}, there is a quasi-isomorphism 
\[
\Phi\colon\mathrm{Cone}(G)\to \frac{G}{U_1=U_{n+1}=U_{2n}},\ \Phi\left(\begin{pmatrix}
g_{11} & g_{12} \\
g_{21} & g_{22}
\end{pmatrix} \right)=p(g_{22}),
\]
where $p\colon G\to G/(U_1=U_{n+1}=U_{2n})$ is the projection.
By Lemma \ref{lem:Imf-CxC} we have the following commutative diagram:
\[
\xymatrix@=50pt{
AD\ar[r]^-f \ar[d]^-{e\circ e} & C_0 \ar[d]^-{\mathcal{D}_{SE}\otimes \mathcal{D}_{NW}} \\
G \ar[r]^-{i\circ (U_n+U_{n+1})} \ar[dr]_-{\overline{i\circ (U_n+U_{n+1})}} & \mathrm{Cone}(G)\ar[d]^-{\Phi}\\
{} & \frac{G}{U_1=U_{n+1}=U_{2n}} 
}
\]
where the vertical maps are quasi-isomorphisms and $i$ is the inclusion into the fourth $G$ of $\mathrm{cone}(G)$ and $\overline{i\circ (U_n+U_{n+1})}$ is the induced map.
Then we have a quasi-isomorphism $\Psi\colon \mathrm{Cone}(f)\to\mathrm{Cone}(\overline{i\circ(U_n+U_{n+1})})$.
Again, applying Lemma \ref{lem:conef-c/f} to $\overline{i\circ(U_n+U_{n+1})}$, there is a quasi-isomorphism $\mathrm{Cone}(i\circ(U_n+U_{n+1}))\to \frac{GC^-(g_1)\otimes_{\mathbb{F}}GC^-(g_2)[U_n,U_{n+1}]}{U_1=U_n=U_{n+1}=U_{2n}}$.
The quotient complex $\frac{GC^-(g_1)\otimes_{\mathbb{F}}GC^-(g_2)[U_n,U_{n+1}]}{U_1=U_n=U_{n+1}=U_{2n}}$ is naturally identified with $\frac{GC^-(g_1)\otimes_{\mathbb{F}}GC^-(g_2)}{U_1=U_{2n}}$.

Finally, we write the quasi-isomorphism $\eta\colon C\to \frac{GC^-(g_1)\otimes_{\mathbb{F}}GC^-(g_2)}{U_1=U_{2n}}$ explicitly.
$\eta$ is the composition of the following maps:
\[
\xymatrix@=21pt{
C=\mathrm{Cone}(f) \ar[r]^-\Psi & \mathrm{Cone}(\overline{i\circ(U_n+U_{n+1})}) \ar[r] & \frac{GC^-(g_1)\otimes_{\mathbb{F}}GC^-(g_2)[U_n,U_{n+1}]}{U_1=U_n=U_{n+1}=U_{2n}} \cong \frac{GC^-(g_1)\otimes_{\mathbb{F}}GC^-(g_2)}{U_1=U_{2n}}.
}
\]
Recall that $C$ is freely generated by $\mathbf{AD}_1(g_\#)\cup\s_0(g_\#)$ (Definition \ref{dfn:C}).
So each term of an element of $C$ can be written as $U_1^{k_1}\cdots U_{2n}^{k_{2n}}\mathbf{x}$.
By tracing the diagrams above, we have that
\begin{itemize}
    \item If $\mathbf{x}\in \mathbf{AD}_1(g_\#)$, then $\eta(\mathbf{x})=0$,
    \item If $\mathbf{x}\in \mathbf{NN}(g_\#)$, then $\eta(\mathbf{x})=0$,
    \item If $\mathbf{x}\in \mathbf{IN}(g_\#)$, then $\eta(\mathbf{x})=((e\circ \mathcal{H}_{Hex})\otimes e)(\mathbf{x})$,
    \item If $\mathbf{x}\in \mathbf{NI}(g_\#)$, then $\eta(\mathbf{x})=(e \otimes (e\circ \mathcal{H}_{Hex}))(\mathbf{x})$,
    \item If $\mathbf{x}\in \mathbf{II}(g_\#)$, then $\eta(\mathbf{x})=(e \otimes e)(\mathbf{x})$,
\end{itemize}
and $\eta(U_1)=\eta(U_n)=\eta(U_{n+1})=\eta(U_{2n})=U_1$ and $\eta(U_i)=U_i$ for $i\neq 1,n,n+1,2n$.

\end{proof}

\section{The Legendrian grid invariant}
\label{sec:Legendrian-invariant}
\begin{dfn}
For a grid diagram $g$, let $\mathbf{x}^+(g)\in\mathbf{S}(g)$ be the state consisting of the northeast corners of each square containing $X$.
Similarly, $\mathbf{x}^-(g)\in\mathbf{S}(g)$ is the state consisting of the southwest corners of each square containing an $X$-marking.
\end{dfn}

\begin{dfn}
\label{dfn:grid-Legendrian-inv}
For a Legendrian knot $\mathcal{K}$, the \textbf{grid Legendrian invariants} $\lambda^+(\mathcal{K})$ and $\lambda^-(\mathcal{K})$ are the homology class $[\mathbf{x}^+(g)], [\mathbf{x}^-(g)]\in GH^-(g)$ respectively, where $g$ is a grid diagram representing $\mathcal{K}$.
\end{dfn}

\begin{proof}[proof of Theorem \ref{thm:connected-Legendrian-invariant}]
We will use the notations in Section \ref{sec:proof-of-main-thm}.
Consider a grid diagram $g_\#$.
Both $\mathbf{x}^+(g_\#)$ and $\mathbf{x}^-(g_\#)$ are in $\mathbf{AD}_1(g_\#)$.
The quasi-isomorphism induced by the inclusion $i\colon C \to GC^-(g_\#)$ sends $[\mathbf{x}^\pm(g_\#)]$ to $[\mathbf{x}^\pm(g_\#)]$.
The quasi-isomorphism $\eta\colon C\to \frac{GC^-(g_1)\otimes_{\mathbb{F}}GC^-(g_2)}{U_1=U_{2n}}$ sends $[\mathbf{x}^\pm(g_\#)]$ to $[(e\circ e)(\mathbf{x}^\pm(g_\#))]=[\mathbf{x}^\pm(g_1)\otimes\mathbf{x}^\pm(g_2)]$.
Then Theorem \ref{thm:invariance-lambda} completes the proof.
\end{proof}

\bibliography{grid}
\bibliographystyle{amsplain} 

\end{document}